\documentclass[12pt]{amsart}
\theoremstyle{plain}
\newtheorem{thm}{Theorem}[section]
\newtheorem{theorem}[thm]{Theorem}

\newtheorem{lemma}[thm]{Lemma}

\newtheorem{proposition}[thm]{Proposition}
\theoremstyle{definition}
\newtheorem{remark}[thm]{Remark}

\newtheorem{question}[thm]{Question}

\numberwithin{equation}{section}

\newcommand{\C}{{\mathbb C}}

\newcommand{\BP}{{\mathbb P}}

\newcommand{\Z}{{\mathbb Z}}

\title [Generalized Kummer manifolds]{No cohomologically trivial 
non-trivial automorphism of generalized Kummer manifolds}

\author{Keiji Oguiso}

\address{Mathematical Sciences, the University of Tokyo, Meguro Komaba 3-8-1, Tokyo, Japan and Korea Institute for Advanced Study, Hoegiro 87, Seoul, 
133-722, Korea}
\email{oguiso@ms.u-tokyo.ac.jp}
\thanks{The author is supported by JSPS Grant-in-Aid (S) No 25220701, JSPS Grant-in-Aid (S) 15H05738, JSPS Grant-in-Aid (B) 15H03611, and by KIAS Scholar Program.}

\dedicatory{Dedicated to Professor Tomohide Terasoma on the occasion of his sixtieth birthday}

\begin{document}

\maketitle

\begin{abstract}
For a hyperk\"ahler manifold deformation equivalent to a generalized Kummer manifold, we prove that the action of the automorphism group on the total Betti cohomolgy group is faithful. This is a sort of 
generalization of a work of Beauville and 
a more recent work of Boissi\`ere, Nieper-Wisskirchen and Sarti, concerning the action of the automorphism group of a generalized Kummer manifold on the second cohomology group. 
\end{abstract}

\section{Introduction}\label{sect1}

Thoughout this note, we work over $\C$. Our main result is Theorem \ref{main}. 

Global Torelli theorem for K3 surfaces (\cite{PS71}, \cite{BR75}, 
see also \cite{BHPV04}) says that the contravariant action 
$$\rho_2 : {\rm Aut}\,(S) \to {\rm GL}\, (H^2(S, \Z))\,\, ;\,\, g \mapsto g^*|_{H^2(S, \Z)}$$ 
is faithful for any K3 surface $S$. On the other hand, Dolgachev (\cite[4.4]{Do84}) and Mukai and Namikawa (\cite{MN84}, \cite{Mu10}) show that there are Enriques surfaces $E$ such that the action $\rho_2 : {\rm Aut}\,(E) \to {\rm GL}\, (H^2(E, \Z))$ is not faithful. Here and hereafter, we denote by 
$${\rm GL}\, (L) := {\rm Aut}_{{\rm group}}\,(L)$$ 
for a finitely generated abelian group $L$, possibly with non-trivial torsion.

Throughout this note, by a hyperk\"ahler manifold, we mean a simply-connected compact K\"ahler manifold $M$ admitting an everywhere non-degenerate global holomorphic $2$-form $\omega_M$ such that $H^0(M, \Omega_M^2) = 
\C \omega_M$. Standard examples of hyperk\"ahler manifolds are the Hilbert scheme 
${\rm Hilb}^n (S)$ of $0$-dimensional closed subschemes 
of length $n$ on a K3 surface $S$, the generalized Kummer manifold 
$K_{n-1}\, (A)$, of dimension $2(n-1) \ge 4$, associated to a 
$2$-dimensional complex torus $A$, and their deformations (\cite[Sections 6,7]{Be83}, see also Section \ref{sect2}). 

Beauville (\cite[Propositions 9, 10]{Be83-2}) 
considered a similar question for hyperk\"ahler manifolds and found the following:

\begin{theorem}\label{beauville}

\begin{enumerate}

\item The action $\rho_2 : {\rm Aut}\,({\rm Hilb}^n (S)) \to {\rm GL}\, (H^2({\rm Hilb}^n (S), \Z))$ is faithful. 

\item The action $\rho_2 : {\rm Aut}\,(K_{n-1}(A)) \to 
{\rm GL}\, (H^2(K_{n-1}(A), \Z))$ 
is not faithful. More precisely, $T(n) \subset {\rm Ker}\, \rho_2$. 
\end{enumerate}
\end{theorem}

Here $T(n) \simeq (\Z/n)^{\oplus 4}$ is the group of automorphisms induced by the group of $n$-torsion points $T[n] := \{a \in A \vert na =0\}$ of $A = {\rm Aut}^0(A)$. 

It is natural and interesting to determine ${\rm Ker}\, \rho_2$ in Theorem \ref{beauville} (2). In this direction, Boissi\`ere, Nieper-Wisskirchen and Sarti (\cite[Theorem 3, Corollary 5 (2)]{BNS11}) found the following complete answer:

\begin{theorem}\label{sarti}
${\rm Ker}\, (\rho_2 : {\rm Aut}\,(K_{n-1}(A)) \to 
{\rm GL}\, (H^2(K_{n-1}(A), \Z))) = T(n) \lhd \langle \iota \rangle$. 
\end{theorem}

Here $\iota$ is the automorphism induced by the inversion $-1$ of $A$ and $T(n) \lhd \langle \iota \rangle$ is the semi-direct product of $T(n)$ and $\langle \iota \rangle$, in which $T(n)$ is normal. 

It is also natural and interesting to ask if the action of ${\rm Aut}\, (K_{n-1}(A))$ on the total cohomology group 
$H^*(K_{n-1}(A), \Z) := \oplus_{k=0}^{4(n-1)} H^k(K_{n-1}(A), \Z)$ is faithful or not. 

Our aim is to answer this question in a slightly more generalized form:

\begin{theorem}\label{main}
The action $\rho : {\rm Aut}\,(Y) \to 
{\rm GL}\, (H^*(Y, \Z))$ is faithful for any hyperk\"ahler manifold $Y$ deformation equivalent to $K_{n-1}(A)$. 
\end{theorem}

First we prove Theorem \ref{main} for $K_{n-1}(A)$. By Theorem \ref{sarti}, it suffices to show that $g^* \vert_{H^*(K_{n-1}(A), \C)} \not= id$ for each $g \in (T(n) \lhd \langle \iota \rangle) \setminus \{id\}$. This is checked in Section \ref{sect3}. We then prove Theorem \ref{main} for any $Y$ in Section \ref{sect4}, by using the density result due to Markman and Mehrotra (\cite{MM12}). In Section \ref{sect5}, among other things, we remark a similar 
result for deformation of the Hilbert scheme of a K3 surface (Theorem \ref{hilb}). 

After posting this note on ArXiv (on 2012), Professor Yuri Tschinkel kindly informed me that the action 
$T(n) \lhd \langle \iota \rangle$ on $K_{n-1}(A)$ extends to a faithful action on any deformation $Y$ of $K_{n-1}(A)$, in such a way that the extended action is trivial on $H^2(Y, \Z)$ (\cite[Theorem 2.1, Proposition 3.1]{HT10}). In particular, this shows that the action 
$$\rho_2 : {\rm Aut}\, (Y) \to {\rm GL}(H^2(Y, \Z))$$ 
is not faithful even if $Y$ is generic.

I should also mention that Theorem \ref{main} is much motivated by the following question asked by Professor Dusa McDuff to me at the conference in Banff (July 2012), while the question itself is still completely open:

\begin{question}\label{McDuff}
Is there an example of a compact K\"ahler manifold $M$ such that the biholomorphic automorphism group is discrete, i.e., ${\rm Aut}^0\, (M) = \{id_M\}$, but with a biholomorphic automorphism $g \not= id_M$ being homotopic to $id_M$ in the group of diffeomorphisms?
\end{question}

{\bf Ackowledgement.} I would like to express my thanks to 
Professor Dusa McDuff for her inspiring question, to Professor Nessim Sibony for invitation to the conference at Banff (July 2012) and Professor Yuri Tschinkel for his interest in this work and a valuable information 
about a paper \cite{HT10}. I also would like to express my thanks to Professor Igor Dolgachev for valuable comments in the first version (in 2012) and for reminding me of this paper in his openning talk at the international conference celabrating Professor Shigeyuki Kond\=o's sixtieth birthday held at Nagoya on December 2017. I also would like to express my thanks to the referee for his/her careful reading.

\section{Preliminaries.}\label{sect2}

In this section, we mainly fix notations we shall use. We follow \cite{Be83} and \cite{Be83-2}. So, $K_n(A)$ in \cite{BNS11} is $K_{n-1}(A)$ in this note.

We refer to \cite[Section 7]{Be83} and \cite[Part III]{GHJ03} for more details on generalized Kummer manifolds 
and basic properties on hyperk\"ahler manifolds. 

Let $A$ be a $2$-dimensional complex torus and let $n$ be an integer such that 
$n \ge 3$. Let ${\rm Hilb}^n(A)$ be the Hilbert scheme of $0$-dimensional closed subschemes of $A$ of length $n$. Then ${\rm Hilb}^n(A)$ is a smooth K\"ahler manifold of dimension $2n$. Let 
$$\nu = \nu_A : {\rm Hilb}^n(A) \to {\rm Sym}^n(A) = A^n/S_n$$ be the Hilbert-Chow morphism. We denote the sum as $0$-cycles by $\oplus$ and the sum 
in $A$ by $+$. Then each element of ${\rm Sym}^n(A)$ is of the form 
$$\oplus_{i=1}^{k} x_i^{\oplus m_i}\,\, .$$
Here $x_i$ are distinct points on $A$ and $m_i$ are positive integers 
such that $\sum_{i=1}^{k} m_i = n$. We have the following surjective morphism 
$$s := s_A : {\rm Sym}^n(A) \to A\,\, ;\,\, \oplus_{i=1}^{k} x_i^{\oplus m_i} \mapsto \sum_{i=1}^{k} m_ix_i\,\, .$$
The generailzed Kummer manifold $K_{n-1}(A)$ is defined by 
$$K_{n-1}(A) := (s \circ \nu)^{-1}(0)\,\, .$$
Note that the morphism
$$s \circ \nu = s_A \circ \nu_A : {\rm Hilb}^n(A) \to A$$
is a smooth surjective morphism such that all fibers are isomorphic (\cite[Section 7]{Be83}). So, $K_{n-1}(A)$ is isomorphic to any fiber of $s_A \circ \nu_A$. 
One can also describe $K_{n-1}(A)$ in a slightly different way, as follows. Let $$A(n-1) := \{(P_1, P_2, \cdots , P_n) \in A^n\, \vert\, \sum_{i=1}^{n} P_i = 0\}\,\, .$$
Then $A(n-1)$ is a closed submanifold of $A^n$ and $A(n-1) \simeq A^{n-1}$. Moreover, $A(n-1)$ is stable under the action of $S_n$ on $A^n$ and 
$${\rm Sym}^n(A) \supset A^{(n-1)} := s^{-1}(0) = 
A(n-1)/S_n\,\, .$$
From this, we deduce that
$$K_{n-1}(A) = \nu^{-1}(A^{(n-1)}) = {\rm Hilb}^n(A) \times_{{\rm Sym}^n(A)} A^{(n-1)}\,\, .$$
Recall that $\dim\, {\rm Def}\, (A) = 4$, while 
$\dim\, {\rm Def}\, (K_{n-1}(A)) = 5$ for $n \ge 3$ and any local deformation 
of a hyperk\"ahler manifold is a hyperk\"ahler manifold (\cite[Section 7]{Be83}). So, there are hyperk\"ahler manifolds which are deformation equivalent to $K_{n-1}(A)$ but are not isomorphic to any generalized Kummer manifold.

From now until the end of this note, we denote by $X := K_{n-1}(A)$ ($n \ge 3$) the generalized Kummer manifold, of dimension $2(n-1)$, associated to a $2$-dimensional complex torus $A$ and by $K := T(n) \lhd \langle \iota \rangle$ the subgroup of ${\rm Aut}\, (K_{n-1}(A))$ defined in the Introduction. We also use the notations introduced in this section freely in the remaining sections.

\section{Proof of Theorem \ref{main} for $K_{n-1}(A)$.}\label{sect3}

In this section, we prove Theorem \ref{main} for $K_{n-1}(A)$. 

First we prove:
\begin{proposition}\label{iota}
Let $g \in K \setminus T(n)$. Then $g^* \vert_{H^3(X, \C)} \not= id$.
\end{proposition}

\begin{proof} Let $(z^1_i, z^2_i)$ ($1 \le i \le n$) 
be the standard global coordinates of the universal cover 
$\C^2_i$ of the $i$-th factor $A_i = A$ of $A^n$. Then the universal cover $\C^{2(n-1)}$ of $A(n-1) \simeq A^{n-1}$ is a closed submanifold of $\C^{2n}$ defined by
\begin{equation}\label{eq21}
z^1_1 + z^1_2 + \cdots + z^1_{n-1} + z^1_{n} = 0\,\, ,\,\, z^2_1 + z^2_2 + \cdots + z^2_{n-1} + z^2_{n} = 0\,\, .
\end{equation}
In particular, $(z^1_i, z^2_i)$ ($1 \le i \le n-1$) give the global coordinates of the universal cover $\C^{2(n-1)}$ of $A(n-1)$. 
Note that $1$-forms $dz^1_i$ and $dz^2_i$ ($1 \le i \le n$) can be regarded as global 
$1$-forms on $A(n-1)$. They satisfy
\begin{equation}\label{eq22}
dz^1_1 + dz^1_2 + \cdots + dz^1_{n-1} + dz^1_{n} = 0\,\, ,\,\, dz^2_1 + dz^2_2 + \cdots + dz^2_{n-1} + dz^2_{n} = 0\,\, ,
\end{equation}
and $\{dz^1_i, dz^2_i (1 \le i \le n-1)\}$ forms a basis of the space of global holomorphic $1$-forms on $A(n-1) \simeq A^{n-1}$.
Consider the following global 
$(2,1)$-form $\tilde{\tau}$ on $A(n-1)$:
$$\tilde{\tau} = dz^1_1 \wedge dz^2_1 \wedge d\overline{z}^2_1 + \cdots + 
dz^1_{n-1} \wedge dz^2_{n-1} \wedge d\overline{z}^2_{n-1} + 
dz^1_n \wedge dz^2_n \wedge d\overline{z}^2_n\,\, .$$
\begin{lemma}\label{form} 
$\tilde{\tau}$ decends to a non-zero element $\tau$ of $H^{2,1}(X)$.
\end{lemma}
\begin{proof} Recall that, for compact K\"ahler orbifolds, the Hodge decompsition is pure and the Hodge theory works in the same way as smooth compact manifolds (\cite{St77}). 

Since $\tilde{\tau}$ is $S_n$-invariant, it descends to 
a global $(2,1)$-form, say $\overline{\tau}$, on the compact K\"ahler orbifold 
$A^{(n-1)}$.  Then $\tau = (\nu \vert_X)^* 
\overline{\tau} \in H^{2,1}(X)$ under $\nu \vert_X : X \to A^{(n-1)}$. It remains to show that $\tau \not= 0$. 
Since $(\nu \vert_X)^*$ is injective, it suffices to show that $\overline{\tau} \not= 0$ in $H^{2,1}(A^{(n-1)})$. For this, it suffices to show that $\tilde{\tau} \not= 0$ in $H^{2,1}(A(n-1))$, as $q^*$ is also injective for the quotient map $q : A(n-1) \to A^{(n-1)}$. 
By Equation (\ref{eq22}), we have
$$\tilde{\tau} = dz^1_1 \wedge dz^2_1 \wedge d\overline{z}^2_1 + \cdots + 
dz^1_{n-1} \wedge dz^2_{n-1} \wedge d\overline{z}^2_{n-1} 
- (\sum_{k=1}^{n-1} dz^1_k) \wedge (\sum_{k=1}^{n-1} dz^2_k) \wedge 
(\sum_{k=1}^{n-1} d\overline{z}^2_k)\,\, .$$
This is the expression of $\tilde{\tau}$ in terms of the standard basis of $H^{2,1}(A(n-1))$. As $n-1 \ge 2$, the term 
$$dz^1_1 \wedge dz^2_2 \wedge d\overline{z}^2_2$$ 
appears with coefficient $-1$ in this expression. Hence $\tilde{\tau} \not= 0$ 
in $H^{2, 1}(A(n-1))$. This proves Lemma \ref{form}. 
\end{proof}

\begin{lemma}\label{form2} 
Let $g \in K \setminus T(n)$ and $\tau \in H^{2,1}(X)$ be as in Lemma \ref{form}. Then $g^* \tau = -\tau$. In particular, $g^*|_{H^3(X, \C)} \not= id$. 
\end{lemma}

\begin{proof} The automorphism $g$ acts equivariantly on $A(n-1) \rightarrow A^{(n-1)} \leftarrow X$. For $\iota$, hence for $g \in K \setminus T(n)$, we have 
$$\iota^*dz_i^q = -dz_i^q\,\, ,\,\,g^*dz_i^q = -dz_i^q\,\, (1 \le i \le n\,\, ,\,\, q=1, 2)\,\, .$$ 
Hence $g^*\tilde{\tau} = -\tilde{\tau}$ by the shape of $\tilde{\tau}$. Thus $g^*\tau = -\tau$. By Lemma \ref{form}, $\tau \not= 0$ in $H^{2,1}(X)$. Hence $g^* \vert_{H^3(X, \C)} \not= id$ as claimed. 
\end{proof}

Lemma \ref{form2} completes the proof of Proposition \ref{iota}. 
\end{proof}

Next we prove: 
\begin{proposition}\label{translation}
Let $a \in T(n) \setminus \{id\}$. 
Then $a^* \vert_{H^*(X, \C)} \not= id$. 
\end{proposition}

\begin{proof} Let $a \in T(n) \simeq (\Z/n\Z)^{\oplus 4}$ be an element of order $p \not= 1$ ($p$ is {\it not} necessarily a prime number). Set $d = n/p$. Then $d$ is a positive integer such that 
$d < n$. We freely regard $a$ also as a torsion element of order $p$ in $A$ and automorphisms of various spaces which are naturally and equivariantly induced by the translation automorphism $x \mapsto x +a$ of $A$. 

We wil show first Lemma \ref{fixed}, Theorem \ref{goettche} and Lemma \ref{lefschetz}, and then we will conclude the proof of Proposition \ref{translation}. 

\begin{lemma}\label{fixed}
The fixed locus $X^a$ consists of $p^3$ connected components $F_i$ 
($1 \le i \le p^3$). Moreover, each $F_i$ is isomorphic to the generalized 
Kummer manifold $K_{d-1}(A/\langle a \rangle)$ associated to the 
$2$-dimensional complex torus $A/\langle a \rangle$. 
\end{lemma}

\begin{proof} Let ${\mathcal S} \subset A$ be a $0$-dimensional closed subscheme of length $n$. As $\langle a \rangle$ acts freely on $A$, the quotient map $\pi : A \to A/\langle a \rangle$ is \'etale of degree $p$. It follows that $a_*{\mathcal S} = {\mathcal S}$ if and only if there is a $0$-dimensional closed subscheme ${\mathcal T} \subset A/\langle a \rangle$ of length $d = n/p$ such that ${\mathcal S} = \pi^*{\mathcal T}$. This ${\mathcal T}$ is clearly unique and we obtain an isomorphism 
\begin{equation}\label{eq30}
{\rm Hilb}^d(A/\langle a \rangle) \simeq ({\rm Hilb}^n(A))^a\,\, ;\,\, {\mathcal T} \mapsto \pi^*{\mathcal T}\,\, .
\end{equation}
Let ${\mathcal S} \in  ({\rm Hilb}^n(A))^a$. Then $\nu({\mathcal S}) \in ({\rm Sym}^n(A))^a$ as well, and $\nu({\mathcal S})$ is then of the form 
$$\nu({\mathcal S}) = \oplus_{i=1}^{k} 
\oplus_{j=0}^{p-1} (x_i + ja)^{\oplus m_i}\,\, $$
(and vice versa).  
Here $\sum_{i=1}^{k}m_i = d$ and all points $x_i + ja$ are distinct. Observe also that 
$$K_{n-1}(A)^a = ({\rm Hilb}^n(A))^a \cap K_{n-1}(A)\,\, .$$ 
As ${\mathcal S} \in ({\rm Hilb}^n(A))^a$ by our choice of ${\mathcal S}$, it follows from the equality above that ${\mathcal S} \in K_{n-1}(A)^a$ 
if and only if ${\mathcal S} \in K_{n-1}(A)$, i.e., ${\mathcal S} \in  ({\rm Hilb}^n(A))^a$ satisfies (by the definition of $K_{n-1}(A)$ and by the shape of $\nu({\mathcal S})$) that
\begin{equation}\label{eq31}
p(m_1x_1 + m_2x_2 + \cdots + m_kx_k + \alpha) = 0
\end{equation}
in $A$. Here $n(p-1)/2 \in \Z$ and $\alpha \in A$ is an element 
such that 
$$p\alpha = (n(p-1)/2) a$$ 
in $A$. We choose and fix such $\alpha$. 

Let $A[p]$ be the group of $p$-torsion points of $A$. Then, Equation (\ref{eq31}) is equivalent to 
\begin{equation}\label{eq32}
m_1x_1 + m_2x_2 + \cdots + m_kx_k + \alpha \in A[p]\,\, .
\end{equation}
Since $a$ is also a $p$-torsion point, Equation (\ref{eq32}) is also equivalent to
\begin{equation}\label{eqn33}
m_1\pi(x_1) + m_2\pi(x_2) + \cdots + m_k\pi(x_k) + \pi(\alpha) \in \pi(A[p]) = A[p]/\langle a \rangle\,\, .
\end{equation} 
Write ${\mathcal S} = \pi^*{\mathcal T}$. Then, Equation (\ref{eqn33}) holds if and only if ${\mathcal T}$ is in the fibers of 
$$s_{A/\langle a \rangle} \circ \nu_{A/\langle a \rangle} : {\rm Hilb}^d(A/\langle a \rangle) \to A/\langle a \rangle$$ 
over $\pi(A[p])$. We have $\vert \pi(A[p]) \vert = p^3$, 
as $a$ is also $p$-torsion. Hence, by the isomorphism (\ref{eq30}), the fixed locus $K_{n-1}(A)^a$ is isomorphic to the union of $p^3$ fibers of $s_{A/\langle a \rangle} 
\circ \nu_{A/\langle a \rangle}$ and  
each fiber is isomorphic to 
$K_{d-1}(A/\langle a \rangle)$ as remarked in Section \ref{sect2}. This completes the proof of Lemma \ref{fixed}.  
\end{proof}
Set 
$$\sigma(n) = \sum_{1 \le b \vert n} b\,\, ,$$ 
the sum of all positive divisors of a positive integer $n$. 
The following fundamental result due to G\"ottche and Soergel 
(\cite[Corollary 1]{GS93}, see also \cite{Go94}, \cite{De10}) is crucial in our proof:

\begin{theorem}\label{goettche}
The topological Euler number $\chi_{\rm top}(K_{n-1}(A))$ of $K_{n-1}(A)$ is 
$n^3 \sigma(n)$. (This is also valid for $n=1$, $2$.)
\end{theorem}

Now we consider the Lefschetz number of $h \in {\rm Aut}\, (X)$:
$$L(h) := \sum_{k=0}^{4(n-1)} (-1)^k\,\, {\rm tr}\,\, h^{*} \vert_{H^{k}(X, \C)}\,\, .$$

\begin{lemma}\label{lefschetz}

(1) If $h \in {\rm Aut}\, (X)$ is cohomologically trivial, then $L(h) = n^3\sigma(n)$. 

(2) $L(a) = n^3\sigma(d)$ for any element $a$ of order $p$ in $T(n) \setminus \{id\}$ with $d = n/p$. 
\end{lemma}

\begin{proof} If $h$ is cohomologically trivial, then ${\rm tr}\, h^{*} \vert_{H^{k}(X, \C)} = b_k(X)$. This implies (1). By the topological Lefschetz fixed point formula, Lemma \ref{fixed} and Theorem \ref{goettche}, we obtain
$$L(a) = \chi_{\rm top}(X^a) = p^3 \chi_{\rm top}(K_{d-1}(A/\langle a \rangle) 
= p^3\cdot d^3\sigma(d) = n^3\sigma(d)\,\, .$$ 
This is nothing but the assertion (2). This proves Lemma \ref{lefschetz}. 
\end{proof}

Since $d \vert n$ and $d \not= n$, it follows that 
$$\sigma(d) \le \sigma(n) -n < \sigma(n)\,\, .$$ 
Hence $a \in T(n) \setminus \{id\}$ is not cohomologically trivial by Lemma \ref{lefschetz}. 
This proves Proposition \ref{translation}. 
\end{proof}

Theorem \ref{main} for $K_{n-1}(A)$ now follows from Theorem \ref{sarti}, Proposition \ref{iota} and Proposition \ref{translation}. This completes the proof of Theorem \ref{main} for $K_{n-1}(A)$.

\section{Proof of Theorem \ref{main}.}\label{sect4}

In this section, we shall prove Theorem \ref{main} for any $Y$. 

Let $\Lambda = (\Lambda, (*, **))$ be a fixed abstract lattice 
isometric to 
$(H^2(K_{n-1}(A), \Z), b)$. Here $b$ is the Beauville-Bogomolov form of $K_{n-1}(A)$ (see eg. \cite[Example 23.20]{GHJ03}). 

Let $Y$ be a hyperk\"ahler manifold deformation equivalent to a generalized 
Kummer manifold $X = K_{n-1}(A)$. Let $g \in {\rm Aut}\, (Y)$ such that 
$g^* \vert_{H^*(Y, \Z)} = id$. We are going to show that $g = id_Y$. 

Let ${\mathcal M}^0$ be the connected component of 
the marked moduli space of ${\mathcal M}_{\Lambda}$, containing $(Y, \eta)$. 
Here $\eta : H^2(Y, \Z) \to \Lambda$ is a marking. 
Huybrechts constructed the marked moduli space ${\mathcal M}_{\Lambda}$ (\cite[1.18]{Hu99}) by patching Kuranishi spaces via local Torelli theorem for hyperk\"ahler manifolds (\cite[Theorem 5]{Be83}, \cite[25.2]{GHJ03}). 
By construction, ${\mathcal M}_{\Lambda}$ is smooth, but highly non-Hausdorff. He also showed that 
the period map 
$$p : {\mathcal M}^0 \to {\mathcal D} = \{[\omega] \in \BP(\Lambda \otimes \C\, \vert (\omega, \omega) = 0 \, ,\, (\omega, \overline{\omega}) > 0 \}$$ is a surjective holomorphic map of degree $1$ (\cite[Theorem 8.1]{Hu99}, see also \cite{Ve09}, \cite{Hu11} for degree and futher development). Let $[\omega] \in {\mathcal D}$. If $p^{-1}([\omega])$ ($\subset {\mathcal M}^0$) is not a single point, then  $p^{-1}([\omega])$ consists of points, being mutually inseparable, corresponding to birational hyperk\"ahler manifolds (\cite[Theorem 8.1]{Hu99}). 

By using the Hodge theoretic Torelli type Theorem (\cite{Ma11}), Markman and Mehrotra (\cite[Theorem 4.1]{MM12}) proved that the marked generalized Kummer manifolds are dense in ${\mathcal M}^0$. Actually they proved the following stronger density result: 

\begin{theorem}\label{markman}
There is a dense subset ${\mathcal D}' \subset {\mathcal D}$ such that if $[\omega] \in {\mathcal D}'$, then any point of $p^{-1}([\omega])$ corresponds to a marked generalized Kummer manifold. 
\end{theorem}

Consider the Kuranishi family $u : {\mathcal U} \to {\mathcal K}$ of $Y$. Here and hereafter we freely shrink ${\mathcal K}$ around $0 = [Y]$. 
Since the Kuranishi family is universal, $g \in {\rm Aut}\, (Y)$ induces automorphisms 
$\tilde{g} \in {\rm Aut}\, ({\mathcal U})$ and $\overline{g} \in {\rm Aut}\, ({\mathcal K})$ such that $u \circ \tilde{g} = \overline{g} \circ u$ 
and 
$\tilde{g}\vert Y = g$. Since ${\mathcal K}$ is locally isomorphic to ${\mathcal D}$ by the local Torelli theorem, the locus 
$${\mathcal K}' \subset {\mathcal K}\,\, ,$$ 
consisting of the point $t$ such that 
$u^{-1}(t)$ is a generalized Kummer manifold, is dense in ${\mathcal K}$. 
This is a direct consequence of Theorem \ref{markman} and 
the construction of ${\mathcal M}^0$ explained above. {\it Here we also emphasize that the density in ${\mathcal M}^0$ is not sufficient to conclude this.} 

From now we follow Beauville's argument (\cite[Proof of Proposition 10]{Be83-2}). 

Let $T_Y$ be the tangent bundle of $Y$. Then, one can take ${\mathcal K}$ as a small polydisk in 
$H^1(Y, T_Y)$ with center $0$. 
As $\omega_Y$ is everywhere non-degenerate, we have an isomorphism
$$H^1(Y, T_Y) \simeq H^1(X, \Omega_Y^1)$$ 
induced by the isomorphism 
$$T_Y \simeq \Omega_Y^1 = T_Y^* \,\, :\,\, v \mapsto \omega_Y(v, *)\,\, .$$ 
As $g$ is cohomologically trivial 
and $H^{2,0}(Y) = H^0(Y, \Omega_Y^2) = \C\omega_Y$, we have $g^*\omega_Y = \omega_Y$ and $g^* \vert_{H^1(Y, \Omega_Y^1)} = id$. Hence by the isomorphism above, we obtain $g^*|_{H^1(Y, T_Y)} = id$, and therefore, $\overline{g} = id_{{\mathcal K}}$. 

Let $t \in {\mathcal K}$ be any point of ${\mathcal K}$. Then, by $\overline{g} = id_{{\mathcal K}}$, the morphism $\tilde{g}$ preserves the fiber $Y_t = u^{-1}(t)$, 
i.e., 
$$\tilde{g}\vert_{Y_t} \in {\rm Aut}\, (Y_t)\,\, .$$ 
Put $g_t := \tilde{g}\vert_{Y_t}$. Then $g_t$ is also cohomologically trivial, because $g_t^*\vert_{H^*(Y_t, \Z)}$ 
is derived from the action of $\tilde{g}$ on the constant system $\oplus_{k=0}^{4(n-1)} R^ku_* \Z$. Then $g_t = id_{Y_t}$ for all $t \in {\mathcal K}'$, as we already proved  
Theorem \ref{main} for $K_{n-1}(A)$ in Section \ref{sect3}. Since ${\mathcal U}$ is Hausdorff 
and $\tilde{g}$ is continuous, it follows that 
$\tilde{g} = id_{\mathcal U}$. Hence 
$g = g_0 = id_Y$ as well. This completes the proof of Theorem \ref{main}. 

\section{A few concluding remarks.}\label{sect5}

In this section, we remark a few relevant facts, which should be known to some experts.

Our first remark is about an anologue of Theorem \ref{main} for a hyperk\"ahler manifold deformation equivalent to the Hilbert scheme ${\rm Hilb}^n\, (S)$ of a K3 surface $S$. 

Markman and Mehrotra (\cite[Theorem 1.1]{MM12}) also proved the strong density result for ${\rm Hilb}^n (S)$ of K3 surfaces $S$. So, the same argument as in Section \ref{sect4} together with Beauville's result 
(Theorem \ref{beauville} (1)) 
implies the following result due to Mongardi \cite[Lemma 1.2]{Mo13}:

\begin{theorem}\label{hilb}
Let $W$ be a hyperk\"ahler manifold deformation equivalent to 
${\rm Hilb}^n (S)$. Then, the action $\rho_2 : {\rm Aut}\, (W) \to {\rm GL}\, (H^2(W, \Z))$ is faithful. 
\end{theorem}

Our second remark is about the fixed locus of symplectic automorphism of finite order. 

In Lemma \ref{fixed}, we described the fixed locus $X^a$. Our description shows that $X^a$ is a disjoint union of smooth hyperk\"ahler manifolds. However, this is not accidental:
\begin{proposition}\label{symplectic} 
Let $(M, \omega_M)$ be a holomorphic symplectic manifold of dimension 
$2d$, i.e., $M$ is a compact K\"ahler manifold and $\omega_M$ is an everywhere non-degenerate holomorphic $2$-form on $M$ (not necessarily unique up to $\C^{\times}$). Let $h \in {\rm Aut}\, (M)$ such that 
$h^*\omega_M = \omega_M$ and $h$ is of finite order $m$. Let $F$ 
be a connected component of the fixed locus 
$M^h = \{P \in M \vert h(P) = P \}$. Then $(F, \omega_M \vert_{F})$ is a holomorphic symplectic manifold (possibly a point). 
\end{proposition}

\begin{proof} $F$ is isomorphic to the intersection of the graph of $h$ and the diagonal $\Delta$ in $M \times M$. So it is compact and K\"ahler, possibly singular. Let $P \in F$. Since $h$ is of finite order, $h$ is locally linearizable at $P$ (see the proof of \cite[Lemma 1.3]{Ka84}). That is, there are local coordinates $(y_1, y_2, \cdots , y_{2d})$
at $P$ such that 
\begin{equation}\label{eq51}
h^*y_i = y_i\,\, (1 \le \forall i \le r)\,\, ,\,\, h^*y_j = c_jy_j\,\, (r+1 \le \forall j \le 2d)\,\, .
\end{equation} Here $c_j \not= 1$ and satisfies $c_j^m = 1$. Then $F$ is locally defined by $y_j = 0$ ($r+1 \le j \le 2d$) in $M$. Hence $F$ 
is smooth. Consider the linear differential map 
$$dh_{P} : T_{M,P} \to T_{M, P}$$ of the tangent space $T_{M, P}$ of $M$ at $P$.
By Equation (\ref{eq51}), we have the decomposition 
\begin{equation}\label{eq52}
T_{M, P} = T_{F, P} \oplus N\,\, .
\end{equation} 
Here $N = \oplus_{j=r+1}^{2d} \C v_j$ for some $v_j$ ($r+1 \le j \le 2d$) 
such that $dh_{P}(v_j) = c_j^{-1}v_j$ and the tangent space $T_{F,P}$ of $F$ at $P$ is exactly the invariant subspace $(T_{M, P})^{dh_{P}}$. Using $h^*\omega_M = \omega_M$, we deduce that
$$\omega_{M, P}(v, v_j) = \omega_{M, P}(dh_{P}(v), dh_{P}(v_j)) = 
\omega_{M, P}(v, c_j^{-1}v_j) = c_j^{-1}\omega_{M, P}(v, v_j)\,\, ,$$
for any $v \in T_{F, P}$ and $v_j$ ($r+1 \le j \le 2d$). As $c_j \not= 1$, it follows that 
$$\omega_{M, P}(v, v_j) = 0$$ 
for all $v \in T_{F, P}$ and $v_j$ with $r+1 \le j \le 2d$. Hence, the decomposition (\ref{eq52}) is orthogonal with respect to $\omega_{M, P}$. As $\omega_{M, P}$ is non-degenerate, 
it follows from the orthogonality of the decomposition that $\omega_{M, P} \vert_{T_{F, P}}$ is also non-degenerate on $T_{F, P}$ (possibly $\{0\}$). Hence $(F, \omega\vert_{F})$ is a smooth symplectic manifold (possibly a point) as well. This completes 
the proof of Proposition \ref{symplectic}. 
\end{proof}

\begin{remark}\label{camere} Proposition \ref{symplectic} is a formal generalization of a result of 
Camere (\cite[Proposition 3]{Ca10}) for a symplectic involution.
\end{remark}

\end{document}